\documentclass{article}

\usepackage{etex}
\usepackage{graphicx,amssymb} %
\usepackage{color}
\usepackage{amsmath}
\usepackage{amsthm}
\usepackage{epsfig}
\usepackage{pstricks}

\usepackage{tikz}
\usepackage{soul}   
\usepackage{xy} \input xy \xyoption{all}

\newcommand*{\DashedArrow}[1][]{\mathbin{\tikz [baseline=-0.25ex,-latex, dashed,#1] \draw [#1] (0pt,0.5ex) -- (1.3em,0.5ex);}}%

\newtheorem{theorem}{Theorem}[section]
\newtheorem{lemma}[theorem]{Lemma}

\newtheorem{example}[theorem]{Example}
\newtheorem{remark}[theorem]{Remark}
\newtheorem{corollary}[theorem]{Corollary}

\def\qed{\hfill\vbox{\hrule\hbox{\vrule\kern3pt\vbox{\kern6pt}\kern3pt\vrule}\hrule}\bigskip}

\makeatletter
\newcommand{\dashedrightarrow}[1][2pt]{%
  \settowidth{\@tempdima}{$\longrightarrow$}\longrightarrow
  \makebox[-\@tempdima]{\hskip-1.5ex\color{white}\rule[0.5ex]{#1}{1pt}}
  \phantom{\longrightarrow}
}
\makeatother

\newcounter{mnotecounter}

\title{On the existence of birational surjective parametrizations of affine surfaces}

\author{J. Caravantes,  J.R. Sendra,  D. Sevilla and C. Villarino}

\begin{document}

\maketitle

\begin{abstract}
In this paper we show that not all affine rational complex surfaces can be parametrized birational and surjectively.
For this purpose, we prove that, if $S$ is an affine complex surface whose projective closure is smooth, a necessary condition for $S$ to admit a birational surjective parametrization from an open subset of the affine complex plane is that the infinity curve of $S$ must contain at least one rational component. As a consequence of this result we provide  examples of affine rational
surfaces that do not admit birational surjective parametrizations.
\end{abstract}



\section{Introduction}
Some  computational problems, of mathematical nature, can be approached by means of algebro-geometric techniques. In these situations, either because of the problem itself directly relates to an algebraic variety or because the problem is translated into an underlying algebraic variety, techniques from computational algebraic geometry are applied. Specially important are those cases where the associated algebraic variety is unirational, since then two different types of representations of the geometric object, namely a set of generators of its ideal or a rational parametrization of it, are available.  Examples of this claim appear in  some practical applications in computer geometric design (see \cite{Bajaj1994a}, \cite{FarinHoschekKim2002a}, \cite{HoschekLasser1993a}, \cite{Sederberg1998a}), where the connection to algebraic geometry is direct. Other examples can be found in the study and solution of algebraic differential equations by means of the analysis of an associated algebraic variety (see e.g.  \cite{FengGao2006}, \cite{Grasseger2014}, \cite{GrasseggerLastraSendraWinkler2016}, \cite{Hubert1996}, \cite{NgoSendraWinkler2015}, \cite{NgoWinkler2010}); for instance, an algebraic non-autonomous first order ordinary differential equation induces an algebraic surface and the existence, and actual computation, of a general rational solution is derived from a birational parametrization of this surface (see \cite{NgoWinkler2010}).

Nevertheless, when dealing with parametric representations one needs to guarantee that certain problematic situations do not appear. An specially important difficulty may occur when dealing with parametrizations that are not surjective. That is, let us work with, say, a rational affine variety $X$, and we take a birational affine  parametrization $f$ of $X$; in other words, a dominant birational map $f:\mathbb{C}^{r} \DashedArrow[->,densely dashed] f(\mathbb{C}^r) \subset X \subset \mathbb{C}^n$, and let us assume that $f$ is not surjective, i.e. $f(\mathbb{C}^r) \subsetneq X$. Then, the feasibility of the use of $f$ depends on whether the  variety property sought, or the information derived from the variety, is only readable from  the non-reachable zone $X\setminus f(\mathbb{C}^r)$ of the algebraic variety.  Example 1.1., in \cite{SendraSevillaVillarino2016a}, illustrates the described difficulty for the problem of computing the distance of a point to an algebraic surface.

When the affine complex variety $X$ is a curve, the problem admits a direct solution, in the sense that $X$ can always be
parametrized birationally and surjectively. Furthermore, in  \cite{AndradasRecio2007a} and \cite{Sendra2002a} one may find algorithms for this purpose. For the case of surfaces,
the problem turns to be more complicated. The question has been approached from two different point of views: either providing one surjective birational affine parametrization  of $X$ (see e.g. \cite{GaoChou1991a}, \cite{PerezSendraVillarino2010} , \cite{SendraSevillaVillarino2015b}), or determining  finitely many birational affine parametrizations $f_1,\ldots,f_s$ of $X$ such that the union of their imagines does cover the whole affine surface, that is $\cup_{i=1}^{s} f_i(\mathbb{C}^2)= X$ (see e.g.  \cite{BajajRoyappa1995}, \cite{GaoChou1991a}, \cite{SendraSevillaVillarino2014a}, \cite{SendraSevillaVillarino2015a}, \cite{SendraSevillaVillarino2016a}). Nevertheless, the following natural question arises: does there exist a surjective birational affine parametrization for every rational affine surface?

In this paper, we answer this question and we prove that, in general, the answer is no. More precisely, in Theorem \ref{main}, we describe the intersection of the projective closure of the given affine surface with the infinity plane under the assumption that the surface can be parametrized surjectively and birationally; in fact, we see that this intersection has to be either smooth or contain at least one rational component. As a consequence, in Example \ref{ex:FermatCubic}, we show that the Fermat cubic surface cannot be parametrized surjectively with a
birational parametrization.

\section{Preliminaries}

In this section, we recall some basic facts that will be used throughout the paper; we refer the reader to \cite{Beauville1996a}, \cite{Hartshorne1977a} for further details. We also include some consequences whose reference is unknown to us.
We work over the complex field $\mathbb{C}$. A variety is an irreducible and reduced projective scheme. We remind the reader that a variety $X$ is \emph{normal} iff the local ring $\mathcal{O}_{X,x}$ is integrally closed for all $x\in X$.  We also recall that any smooth variety is normal.

Some classic results:

\begin{theorem}\label{DimFibra}\cite[Exercise II.3.22(c)]{Hartshorne1977a}
Let $f:X\to Y$ be a surjective morphism of schemes. Then, the dimension of the general fiber of $f$ is $\dim X-\dim Y$.
\end{theorem}

\begin{theorem}\label{IndetAfin}\cite[Exercise I.3.20]{Hartshorne1977a}
Let $X$ be a quasi projective normal surface. Let
$f:X\DashedArrow[->,densely dashed]\mathbb{A}^N$ be a rational map,
whose indeterminacy locus is finite. Then $f$ is a regular morphism.
\end{theorem}

\begin{theorem}\label{FundLoc}\cite[Lemma V.5.1]{Hartshorne1977a}
Let $f:X\DashedArrow[->,densely dashed]\mathbb{P}^N$ be a birational map. If $X$ is normal, the fundamental locus of $f$ has codimension at least 2 in $X$.
\end{theorem}

\begin{theorem}\label{BlowUp}\cite[Theorem II.7]{Beauville1996a}
Let X be a surface. Let $f:X\DashedArrow[->,densely dashed]\mathbb{P}^N$ be a rational map. Then there exists a commutative diagram
\[ \xy
	(0,12)*+{Y}="Y";
	(-15,0)*++{X}="X";
	(15,0)*++{\mathbb{P}^N}="PN";
	{\ar_{\displaystyle g} "Y"; "X"};
	{\ar^{\displaystyle h} "Y"; "PN"};
	{\ar@{-->}^{\displaystyle f} "X"; "PN"};
\endxy \]
where $g$ is a composite of blowups and $h$ is a morphism.
\end{theorem}

\begin{corollary}\label{ImDim1}
In the hypotheses of Theorem \ref{BlowUp}, suppose that $X$ is normal. Then, for any fundamental point $P$ of $f$, $h(g^{-1}(P))$ is a connected finite union of rational curves.
\end{corollary}
\begin{proof}
Let $E$ be the connected component $g^{-1}(P)$ of the exceptional divisor. Since $g$ is a composite of blowups, $E$ is a connected union of irreducible curves, all of them isomorphic to $\mathbb{P}^1$.

If $h(g^{-1}(P))=h(E)$ is not a connected finite union of rational curves, due to the connectedness of $E$ and the fact that it is 1-dimensional, $h(E)$ must be a single point $Q\in \overline{f(X)}\subset\mathbb{P}^N$. Taking an affine neighbourhood $U\subset\mathbb{P}^N$ of $Q$, we have that $V=f^{-1}(U)\cup\{P\}$ is neighbourhood of $P$ in $X$. Applying Theorem \ref{IndetAfin} to $f|_{V}$, we have that $f$ can be extended to $P$, which contradicts the fact that it is a fundamental point.
\end{proof}

\begin{theorem}\label{Castelnuovo}\cite[Corollary V.5.4 and Theorem V.5.7]{Hartshorne1977a}(Castelnuovo's criterion of contractibility)
Let $X$ be a smooth surface and $C$ an irreducible curve in $X$. There exists a smooth surface $Y$ and a birational morphism $f:X\to Y$ contracting $C$ to a point iff $C\simeq \mathbb{P}^1$ and $C^2=-1$. In such case, $f$ is the blowup of the point $f(C)\in Y$.
\end{theorem}

\begin{remark}\label{ClasifBirat}
Theorem \ref{Castelnuovo} says that any birational morphism between nonsingular surfaces is a composite of blowups, each one being the blowup of a closed point.
\end{remark}



\begin{theorem}\label{ZMT}(Zariski's Main Theorem, see e.g. \cite[Corollary III.11.4]{Hartshorne1977a})
Let $f:X\to Y$ be a birational projective morphism between irreducible and reduced varieties. Suppose $Y$ to be normal. Then, for any $y\in Y$, $f^{-1}(y)$ is connected.
\end{theorem}

\section{Surjective parametrizations of affine surfaces}

This section is devoted to proving the following result. 

\begin{theorem}\label{main}
Let $f:\mathbb{C}^2\DashedArrow[->,densely dashed] \mathbb{C}^N$ be a rational map. Let $S$ be the Zariski closure of $f(\mathbb{C}^2)$ in $\mathbb{C}^N$, and suppose that $f$ is birational and surjective onto $S$. Let $\overline{S}$ be the Zariski closure of $S$ in $\mathbb{P}^N$ and $S_\infty=\overline{S}-S$ the infinite hyperplane section.
If $\overline{S}$ is smooth, then $S_\infty$ has at least one rational component.
\end{theorem}

To prove Theorem \ref{main}, we consider Theorem \ref{BlowUp} and get the commutative diagram

\begin{equation}\label{castillo}
\xy
	(0,12)*++{Y}="Y";
	(-15,0)*+++{\mathbb{P}^2}="P2";
	(-15,-10)*+++{\mathbb{C}^2}="A2";
	(15,0)*+++{\overline{S}}="overS";
	(15,-10)*+++{S}="S";
	(30,0)*++{\mathbb{P}^N}="PN";
	(30,-10)*++{\mathbb{C}^N}="AN";
	{\ar_{\displaystyle g} "Y"; "P2"};
	{\ar^{\displaystyle h} "Y"; "overS"};
	{\ar@{-->}^{\displaystyle \overline{f}} "P2"; "overS"};
	{\ar@{-->}^{\displaystyle f} "A2"; "S"};
	(22,0)*{\hookrightarrow};
	(22,-10)*{\hookrightarrow};
	(-16,-5)*{\rotatebox[origin=c]{90}{$\hookrightarrow$}};
	(15,-5)*{\rotatebox[origin=c]{90}{$\hookrightarrow$}};
	(29,-5)*{\rotatebox[origin=c]{90}{$\hookrightarrow$}};
\endxy
\end{equation}
We also establish some more notation. We denote by $F(\overline{f})$  the (finite, by Theorem \ref{FundLoc}) fundamental locus of $\overline{f}$, and  by $L_\infty=\mathbb{P}^2\setminus\mathbb{C}^2$  the infinity line of the plane.

We prove two lemmas before attacking Theorem \ref{main}.

\begin{lemma}\label{FundAInf}
In the conditions of Theorem \ref{main}, and with the notation above, then it holds that $h(g^{-1}(F(\overline{f}))\subset S_\infty$.
\end{lemma}

\begin{proof}
We know by Theorem \ref{FundLoc} that $F(\overline{f})$ is a finite set. Let $P\in F(\overline{f})$ be one of its elements and let $E=g^{-1}(P)$ be the connected component of the exceptional divisor corresponding to $P$. Let $C=h(E)$. We know by Corollary 2.5 that $C$ is a finite union of rational curves. By Theorem \ref{DimFibra}, any general point $Q\in C$ satisfies that $h^{-1}(Q)$ is 0-dimensional (otherwise, $h^{-1}(C)$ would be a surface in $Y$, contradicting the birationality of $f$).

On the other hand, Theorem \ref{ZMT} implies that $h^{-1}(Q)$ is connected. Therefore, $h^{-1}(Q)$ is a single point $P\in Y$. However, since $Q\in C=h(E)$, this means that $P\in E$, so $Q\not\in f(\mathbb{C}^2)$. Since $f$ is surjective onto $S$, this means that $Q\in h(Y)\setminus S=\overline{S}\setminus S=S_\infty$. As this happens for general $Q\in C$, we have that $C\subset S_\infty$.

This is valid for any $P\in F(\overline{f})$, so the proof is completed.
\end{proof}

\begin{lemma}\label{InfAInf}
In the conditions of Theorem \ref{main}, and with the notation above, it holds that $h(g^{-1}(L_\infty))\subset S_\infty$ and
\begin{enumerate}
\item  if $\overline{f}(L_\infty)$ is a curve, then $S_\infty$ contains a curve,
\item  if $\overline{f}$ contracts $L_\infty$, then $S_\infty$ contains at least two rational curves.
\end{enumerate}
\end{lemma}
\begin{proof}
Let us consider two possibilities:
\begin{enumerate}
\item If $\overline{f}(L_\infty)$ is a curve $C$, by Theorem \ref{DimFibra}, any general point $Q\in C$ satisfies that $h^{-1}(Q)$ is 0-dimensional (otherwise, $h^{-1}(C)$ would be a surface in $Y$, contradicting the birationality of $f$). On the other hand, Theorem \ref{ZMT} implies that $h^{-1}(Q)$ is connected. Therefore, $h^{-1}(Q)$ is a single point $P\in Y$.

However, since $Q\in C=\overline{f}(L_\infty)$, we have that $g(h^{-1}(Q))=g(P)\in L_\infty$, and then $Q\not\in f(\mathbb{C}^2)=S$. This means $Q\in h(Y)\setminus S=\overline{S}\setminus S=S_\infty$. The generality of $Q$ means that $C\subset S_\infty$.
\item If $\overline{f}$ contracts $L_\infty$, then Theorem \ref{Castelnuovo} means that the strict transform of $L_\infty$ by $g$ is a $-1$ curve. This is only possible if $g$ blows up two points $P_1$, $P_2$ in $L_\infty$. By Lemma \ref{FundAInf} and Corollary \ref{ImDim1}, $h(g^{-1}(P_i)),\,i=1,2$, is a nonempty finite union of rational curves and it is contained in $S_\infty$. Let $C$ be a rational curve contained in $h(g^{-1}(P_1))$. For the general $Q\in C$, Theorems \ref{DimFibra} and \ref{ZMT} say that the preimage $h^{-1}(Q)$ is 0-dimensional and connected, so it is a point $P\in g^{-1}(P_1)$. Since $g^{-1}(P_2)\cap g^{-1}(P_1)=\emptyset$, because a point cannot have two images, we have that $C$ is not contained in $h(g^{-1}(P_2))$. This means that $S_\infty$ has at least two components, one in $h(g^{-1}(P_1))$ and the other in $h(g^{-1}(P_2))$.
\end{enumerate}
\end{proof}

\begin{proof} \emph{(of Theorem \ref{main})}
It is a direct consequence of Lemma \ref{InfAInf}.
\end{proof}

\begin{remark}
The smoothness of $\overline{S}$ is used any time we apply Zariski's Main Theorem or Castelnuovo's Criterion of Contractibility.
\end{remark}

We now restate Theroem \ref{main} in a way we can use to prove inexistence results easily:

\begin{theorem}\label{twisted_main}
Let $\overline{S}\subset\mathbb{P}^N$ a projective rational smooth surface. Let $S_\infty$ the intersection of $S$ with the infinity hyperplane of $\mathbb{P}^N$ and let $S$ be $\overline{S}-S_\infty$. Suppose that none of the components of $S_\infty$ is a rational curve. Then there does not exist any rational map $f:\mathbb{C}^2\DashedArrow[->,densely dashed]S$ both birational and surjective.
\end{theorem}
\begin{proof}
It is just a restatement of Theorem \ref{main}.
\end{proof}

Next result is useful to study cases where Theorem \ref{main} or \ref{twisted_main} are not applicable

\begin{corollary}\label{BlowInf}
In the conditions of Theorem \ref{main}, and with the notation above,
\begin{enumerate}
\item if $F(\overline{f})\cap L_\infty\ne\emptyset$, then $S_\infty$ contains at least two rational components.
\item if $S_\infty$ contains just one rational component, then $\overline{f}$ is a regular morphism.
\end{enumerate}
\end{corollary}

\begin{proof} First of all, we observe that the first statement is a particular case of the second one: if $F(\overline{f})\cap L_\infty\ne\emptyset$, then $f$ is not a regular morphism. So, by the second statement, we have that $S_\infty$ cannot contain just one rational component (and it contains at least one by Theorem \ref{main}), so it contains at least two of them. To prove the second statement, let $C$ be that only rational component in $S_\infty$. Lemma \ref{InfAInf} says that $\overline{f}$ does not contract $L_\infty$, and that $\overline{f}(L_\infty)$ is an open subset of $C$. By Theorems \ref{DimFibra} and \ref{ZMT}, for a general $Q\in C=\overline{f}(L_\infty)$, $h^{-1}(Q)$ is 0-dimensional and connected, so it is a point $P$. Due to (\ref{castillo}) being commutative, we have $g(P)\in L_{\infty}$.

On the other hand, Lemma \ref{FundAInf}, says that $h(g^{-1}(F(\overline{f}))\subset S_\infty$ and Corollary \ref{ImDim1} says that $h(g^{-1}(F(\overline{f}))$ is a (possibly empty) finite union of rational curves. This means that $h(g^{-1}(F(\overline{f}))\subset C$. However, the last paragraph implies that the general $P\in C$ is not in $h(g^{-1}(F(\overline{f}))$. Since $C$ is irreducible and $h(g^{-1}(F(\overline{f}))$ is either empty or 1-dimensional, we have that $F(\overline{f})=\emptyset$, so $\overline{f}$ is regular.
\end{proof}

\section{Examples and inexistence results}

The next example illustrates that the conditions given by Theorem \ref{main} are sharp. We find a family of examples where just one rational component at the infinity is enough to have a surjective parametrizations from $\mathbb{C}^2$.

\begin{example}
Consider the $d$-th Veronese embedding $v_d:\mathbb{P}^2\to\mathbb{P}^N$, where $N=\frac{(d+2)(d+1)}{2}$, given by all degree $d$ monomials: $v_d(x_0:x_1:x_2)=(x_0^d:x_0^{d-1}x_1:\ldots:x_2^d)$. Let $C\subset\mathbb{P}^2$ be a curve of degree $d-1$ and $L_\infty$ be the infinity line $x_0=0$. Then the ideal of $C\cup L_\infty$ is given by a homogeneous polynomial of degree $d$. This means that $v_d(C\cup L_\infty)$ is the intersection of $\widetilde{S}=v_d(\mathbb{P}^2)$ with a hyperplane $H$. Therefore, if we compose $v_d$ with the suitable automorphism of $\mathbb{P}^N$ that takes $H$ to the infinity, we get a parametrization $f$ of the image $\overline{S}$ such that $S_\infty=f(C\cup L_\infty)$. This means that $S=\overline{S}\setminus S_\infty$ is covered by $\mathbb{P}\setminus(C\cup L_{\infty})=\mathbb{C}^2\setminus C$, and $f_{\mathbb{C}^2\setminus C}$ is an isomorphism.
\end{example}

The next example is one of the main motivations of this paper: finding examples of rational affine surfaces that do not admit birational surjective parametrizations.

\begin{example}\label{ex:FermatCubic}
Consider $S$ to be the Fermat Cubic surface, given by the equation $x^3+y^3+z^3=1$. The intersection with the infinity plane is given by the equation $x^3+y^3+z^3=0$ in projective coordinates. It is a smooth cubic curve, so it is not rational. Therefore, by Theorem \ref{twisted_main}, although $S$ is rational and $\overline{S}$ is smooth, it is impossible to parametrize $S$ surjective and birationally from any open subset from $\mathbb{C}^2$.
\end{example}

Finally, Corollary \ref{BlowInf} provides some ideas to prove inexistence results for other surfaces.

\begin{example}
Consider $S$ to be a smooth quadric hypersurface. There are only two possibilities for $S_\infty$. If $S_\infty$ has singularities, then it must consist in two lines. Then the structure of $\overline{S}$ as $\mathbb{P}^1\times\mathbb{P}^1$ gives the affine part the structure of $\mathbb{C}^1\times\mathbb{C}^1\simeq\mathbb{C}^2$. We would have the well-known parametrization of the paraboloid.

However, if the infinity curve is nonsingular, then it is a conic, which is rational. If there were a surjective birational map $f:\mathbb{C}^2\DashedArrow[->,densely dashed] S$, Corollary \ref{BlowInf} implies that $f$ could be regularly extended to a morphism $\overline{f}:\mathbb{P}^2\to\overline{S}$. Since $f$ is surjective, $\overline{f}$ would be surjective too. However, one can choose two lines in $\overline{S}\simeq\mathbb{P}^1\times\mathbb{P}^1$ whose intersection is empty. Their preimages in $\mathbb{P}^2$ would be two curves with no common points, which is impossible for an algebraically closed field. This implies that it is impossible to find a surjective birational parametrization of the hyperboloid.
\end{example}

\noindent \textbf{Acknowledgments.}
Authors partially supported by
Ministerio de Econom\'{\i}a y Competitividad and the European Regional Development
Fund (ERDF), under the Project MTM2014-54141-P; and by Junta de Extremadura and FEDER funds (group FQM024).


\bigskip

\noindent Jorge Caravantes, Research Group GVP \\ Dpto. de \'Algebra, Universidad Complutense de Madrid \\ Plaza de Ciencias 3, 28040 Madrid, Spain \\ Email: \texttt{jcaravan@mat.ucm.es}

\medskip

\noindent J. Rafael Sendra, Research Group ASYNACS \\ Dpto. de F\'isica y Matem\'aticas, Universidad de Alcal\'a \\ 28871 Alcal\'a de Henares (Madrid), Spain \\
E-mail: \texttt{Rafael.sendra@uah.es}

\medskip

\noindent David Sevilla, Research group GADAC \\ Centro U. de M\'erida, Universidad de Extremadura \\ Av. Santa Teresa de Jornet 38, 06800 M\'erida (Badajoz), Spain \\
E-mail: \texttt{sevillad@unex.es}

\medskip

\noindent Carlos Villarino, Research Group ASYNACS \\ Dpto. de F\'isica y Matem\'aticas, Universidad de Alcal\'a \\ 28871 Alcal\'a de Henares (Madrid), Spain \\
E-mail: \texttt{Carlos.villarino@uah.es}

\end{document}